\documentclass[11pt]{article}
\usepackage{geometry}
\usepackage[parfill]{parskip}

\usepackage{graphicx}

\usepackage{amssymb}

\usepackage{epstopdf}

\usepackage{amsmath,amsthm,amscd,amssymb}
\usepackage{latexsym}
\numberwithin{equation}{section}

\theoremstyle{plain}
\newtheorem{theorem}{Theorem}[section]
\newtheorem{lemma}[theorem]{Lemma}

\theoremstyle{definition}

\theoremstyle{remark}

\newtheorem{case[theorem]}{Case}

\title{A note on the slightly supercritical Navier Stokes equations in the plane}
\author{Nets Hawk Katz and Andrew Tapay}

\begin{document}

\maketitle

\begin{abstract} We produce a new proof of Tao's result on the slightly supercritical
Navier Stokes equations. 
Our proof  has the advantage that it works in the plane
while Tao's proof works
only in dimensions three and higher. We accomplish this by studying the problem as a
system of differential inequalities
on the $L^2$ norms of the Littlewood Paley decomposition, along the lines of Pavlovic's
proof of the Beale-Kato-Majda
theorem.

\end{abstract}

\section{Introduction}

Our goal is to study the slightly supercritical Navier Stokes equations. We will work in
an arbitrary dimension
${\bf R}^d$, although our results are only new when $d=2$. The system we are dealing with
then is
\begin{equation}\label{NS1}{\partial u \over \partial t} + u \cdot \nabla u= -D^2 u  +
\nabla p, \end{equation}
together with the condition of incompressibility
\begin{equation}\label{NS2} \nabla \cdot u =0, \end{equation}
and initial conditions
\begin{equation}  u(x,0) =u_0(x,0). \end{equation}

Here $D$ represents a Fourier multiplier with symbol  ${|\xi|^{d+2 \over 4} \over (\log
(2+|\xi|) )^\gamma}$
with $\gamma \leq {1 \over 4}.$ The aim of this note is to prove:

\bigskip

\begin{theorem} \label{main} The above system with smooth, compactly supported initial
conditions has a global in time
smooth solution. \end{theorem}

This theorem is only new in the case $d=2$. For other dimensions, the results
can be found in
\cite{Tao}.

We remark briefly on the history of the problem. The slightly supercritical Navier Stokes equations
were introduced by Tao in \cite{Tao}. The goal was to quantify exactly how close standard techniques
are to proving global solvability of the standard Navier Stokes equations. The conclusion was that
in dimensions higher than two, these techniques do not get very close but get a little closer
than is suggested by the standard notion of criticality. This generalized a similar observation about
slightly supercritical wave equations with radial symmetry established in \cite{Tao3}. Since Tao's
paper \cite{Tao}, various authors have approached Tao's result in slightly different settings
and with different points of view. (See e.g. \cite{DKV}, \cite{CS}.)

We remark apologetically that our own result is not physically well motivated. In the plane,
the standard Navier Stokes equations are critical and so the slightly supercritical equation
have less dissipation than the standard ones. It is curious however that Tao's argument seemed
not to cover this case.

Our own approach is inspired by Remark 1.2 in Tao's paper \cite{Tao}. There it is shown that if all the
energy of a flow
is concentrated in a single scale, one should expect a result up to $\gamma= {1 \over
2}$. Of course, the
energy need not be concentrated in a single scale. Tao's method of making the heuristic
precise involves
some Sobolev embedding type results that reach an endpoint at $d=2$ where they fail to
work. We study the
scales somewhat more precisely by estimating differential inequalities on the $L^2$ norms
of
Littlewood Paley components of the flow. This is an adaptation of an argument of Pavlovic
for the
classic Beale-Kato-Majda theorem which may be found in her thesis. \cite{Pav} We see that
the reason
that one cannot get an exponent stronger than ${1 \over 4}$ is that one can't rule out
that the singularity
happens merely by making successive scales pass supercriticality without their attaining the
majority of the available
energy.

{\bf Acknowledgements:}  The first author is partially supported by NSF grant DMS-1001607
and a fellowship
from the Guggenheim foundation. He would like to thank Terry Tao for suggesting the
problem. We would also like to thank Vlad Vicol for a helpful discussion.

\section{Littlewood Paley trichotomy}

We introduce a standard Littlewood-Paley decomposition. Let $\phi: {\bf R}^+
\longrightarrow {\bf R}$, be a smooth function
so that $\phi(x)=1$ for $0<x<1$, and $\phi(x)=0$ for $x>2$. We then define $P_0$ to be
the multiplier operator whose symbol is $\phi(|\xi|)$.  We define for $j>0$ that $P_j$ is
the multiplier
operator whose symbol is $\phi( {|\xi| \over 2^{j} })  - \phi( {|\xi| \over 2^{j-1}} )$.
Note that the identity
is the sum of the $P_j$'s since the sum of the symbols telescopes. To simplify notation,
we define $P_k=0$ when
$k$ is negative.

We are interested in the development of $||P_j u||_{L^2}$ over time assuming that $u$
satisfies
the slightly supercritical Navier Stokes equations. Letting $\omega=\nabla \times u$ be
the vorticity,
we see it is entirely equivalent to study the development of $||P_j \omega||_{L^2} \sim
2^j ||P_j u||_{L^2}.$
The norms are equivalent because of the divergence free property of $u$. Passing to the
vorticity helps us
since it eliminates the pressure which is secretly a non-linear term in $u$. We obtain
the vorticity
form of the slightly supercritical Navier Stokes equation:

$${\partial \omega  \over \partial t} +  (\nabla \times (u \cdot \nabla) u)  + (u \cdot
\nabla) \omega
= -D^2 \omega.$$

Here the nonlinear part has broken into two terms, commonly referred to as the vortex
stretching term, which
vanishes when $d=2$ and the advection term. (For our purposes, both will be of equal
strength. Applying a
Littlewood-Paley component with $\langle,\rangle$ denoting $L^2$ inner product in space,
we get

\begin{equation} \label{terms}  \langle {\partial P_j \omega  \over \partial t}, P_j
\omega  \rangle
=- \langle P_j((\nabla \times (u \cdot \nabla) u)  + (u \cdot \nabla) \omega), P_j \omega
 \rangle
 - \langle D^2 P_j \omega, P_j \omega \rangle. \end{equation}

Now by estimating the right hand side, we get bounds for the rate of change in time of
$||P_j \omega||_{L^2}$.

\bigskip

\begin{lemma} \label{trick} With $\omega$ smooth and satisfying the vorticity form of the
slightly supercritical Navier Stokes
equations, there is a universal constant $C>0$ so that we have the estimate
 
\begin{align} \label{est}
{d \over dt}  ||P_j \omega||_{L^2} &\leq   C \left(
\left( \sum_{k \leq j+5}  2^{dk \over 2} ||P_{k} \omega||_{L^2} \right)
\sum_{\alpha=-5}^5 ||P_{j+\alpha} \omega||_{L^2}
+\sum_{k \geq j} 2^{{dj \over 2}} \sum_{\alpha=-5}^5 ||P_k \omega||_{L^2} ||P_{k+\alpha}
\omega||_{L^2} \right ) \\
&- {2^{(d+2) j \over 2} \over j^{2 \gamma} } ||P_j \omega||_{L^2}. \end{align}

\end{lemma}

Lemma \ref{trick} is, by now, a completely standard application of the Littlewood Paley
trichotomy see \cite{Tao2}.
All terms except for the dissipation term appear in Pavlovic's proof of Beale-Kato-Majda
see \cite{Pav}.

The last term of the inequality \ref{est}, comes from the dissipation, namely the last
term of equation \ref{terms}.
The remaining terms on the right hand side come from estimation of the nonlinear term,
namely the first term
on the right in \ref{terms}.

The nonlinear term is estimated by breaking each appearance of $u$ and $\omega$ on the
left hand side
of the inner product into a sum of its Littlewood Paley components. We observe that
purely on frequency
support grounds, only three types of terms survive: the low-high terms, high-low terms,
and high-high terms.
(Since the frequency of the right hand side of the inner product in \ref{terms} is held
fixed at around $2^j$, it is
the only region of frequency support of the product on the left hand side that we need
concern ourselves with.)
We estimate each term, an integral of a product of three components by controlling one
$L^{\infty}$ norm and
two $L^2$ norms. We always use the $L^{\infty}$ norm of the lowest frequency component.
We replace the $L^{\infty}$
norm by an $L^2$ norm using Bernstein's inequality. The right-most term of
the first line of \ref{est} comes from the high-high part. (We've used none of the
structure of the equation and
in fact would get a better estimate from applying a div-curl lemma, but we don't need
this.)

The other terms in the first line of the right hand side of \ref{est} come from the
low-high and high-low terms.
Most of these are quite straightforward. But there is one type of term, coming from the
advection part, which has
to be controlled using a commutator inequality. Namely, we have to obtain estimates on
terms of the form
$$\langle P_l u \cdot \nabla  P_{j \pm 1} \omega , P_j \omega \rangle,$$
with $l$ significantly smaller than $j$. The apparent difficulty here is that we have an
extra derivative falling
on the relatively high frequency part $P_{j \pm 1 } \omega$. We observe however that if
$Q$ is a multiplier supported
at frequency approximately $2^j$ (whose multiplier has derivative bounded by $2^{-j}$)
then we can estimate
$$\langle [Q,P_l u \cdot  \nabla] f,g\rangle \lesssim  2^l ||P_l u||_{L^{\infty}}
||f||_{L^2} ||g||_{L^2}.$$
This commutator idea allows us to exploit the fact that $P_l u \cdot \nabla$ is
antisymmetric because of the
div-free property of $P_l u$. To wit: If we let
$$Q=\sum_{\beta=-3}^3 P_{j+\beta}$$
we can calculate
\begin{align} 
   &\langle P_l u \cdot \nabla  P_{j \pm 1} \omega , P_j \omega \rangle  \\
  =&  \langle P_l u \cdot \nabla  P_{j \pm 1} Q \omega , P_j Q \omega \rangle \\
  =& \langle P_j P_{j \pm 1} P_l u \cdot \nabla Q \omega, Q \omega \rangle
    + \langle  [P_{j \pm 1},P_l u \cdot \nabla] Q \omega, P_j Q \omega \rangle
  \end{align} 
The second term is harmless because it is a commutator and the first term is harmless
because $P_j P_{j \pm 1} P_l u
\cdot \nabla$ is almost skew adjoint - namely:
$$\langle P_j P_{j \pm 1} P_l u \cdot \nabla Q \omega ,Q \omega \rangle
=-\langle P_j P_{j \pm 1} P_l u \cdot \nabla Q \omega ,Q \omega \rangle  +
\langle Q \omega,  [P_{j \pm 1} P_j, P_l u \cdot \nabla] Q \omega \rangle.$$
Thus we can solve for the noncommutator term in terms of the commutator term.

\bigskip

We will now use Lemma \ref{trick} as a black box.

\section{Main Argument}

We are now ready to proceed with the main part of the argument. The idea very much follows
remark 1.2 of \cite{Tao}. Namely, we will produce a measure of progress towards blow-up 
and
we will show that infinite progress towards blow-up requires infinite dissipation of 
energy
thus leading to a contradiction.

We introduce some notation. We define
$$b_j(t)=2^{{(d+2) j \over 2}} ||P_j u(t)||_{L^2}.$$
Thus by Bernstein's inequality, we have that $b_j(t)$ controls $||P_j 
\omega(t)||_{L^{\infty}}$.
We let
$$c(t)=\sum_{j=0}^{\infty}  b_j(t).$$
One may note that $c(t)$ is the norm of $u$ in the Besov space $B^{{d+2 \over 2},1}_2$. 
This is the
norm which implicitly controls growth in the Beale-Kato-Majda argument. Our argument will 
primarily
consist in showing that $c(t)$ remains bounded.

We rewrite Lemma \ref{trick} in terms of the quanitities $b_j(t)$.

\bigskip

\begin{lemma} \label{best} Under the hypotheses of Theorem \ref{main}, there is a 
universal constant $C>0$
so that we have the system of inequalities
\begin{equation} \label{est2}
{db_j \over dt}   \leq   C \left(
\left( \sum_{k \leq j+5}  b_k(t) \right)
\sum_{\alpha=-5}^5 b_{j+\alpha}(t)
+\sum_{k \geq j} 2^{{d(j-k)}} \sum_{\alpha=-5}^5 b_k(t) b_{k+\alpha}(t)
\right )
- {2^{(d+2) j \over 2} \over j^{2 \gamma} } b_j(t).
\end{equation}
\end{lemma}

Let $E$ denote $||u_0||_{L^2}$. Because energy is dissipating, we have for each $j$ that
\begin{equation} \label{energy} b_j(t) \leq E  2^{(d+2) j \over 2} .\end{equation}

We assume that $c(t)$ goes to infinity. We denote by $t_k$, the first time at which $c(t) 
\geq 2^k$. We
are interested in studying what happens between $t_k$ and $t_{k+1}$.

Clearly for $t_{k} < t < t_{k+1}$, we have that by the definition of $c(t)$ for any $j$ the estimate
\begin{equation} \label{trivial}
b_{j}(t) \leq 2^{k+1}.
\end{equation}

However for $j$ sufficiently large, the dissipation term begins to dominate the others. We define
$j_{k,u}$ to be the solution for $j$ of the equation
$$2^k = {2^{{(d+2)j \over 2}} \over j^{\gamma}}.$$

We prove the following barrier estimate. Roughly speaking, it says that when the Besov norm $c(t)$ is small,
we have that $b_j(t)$ is not increasing for large $j$ since dissipation dominates the nonlinear term. However
in light of lemma \ref{best}, we see that shrinkage from dissipation is only in proportion to $b_j(t)$ while
growth from the nonlinear term can come from neighboring Littlewood Paley pieces. The barrier estimate comes
from examining $b_j(t)$ only when it is large compared to its neighbors.

\bigskip

\begin{lemma} \label{barrier} For any $t$  with $t_{k} < t < t_{k+1}$, we have the estimate
\begin{equation} \label{bar} b_j(t) \leq C 2^k  2^{{j_{k,u} - j \over 10}},\end{equation}
with $C$ a universal constant.
\end{lemma}

\begin{proof}  We observe that in light of the inequality \ref{trivial}, the barrier estimate \ref{bar}
is only nontrivial for $k> j_{k,u} + c$, where $c$ is a constant that depends only on $C$
 in \ref{bar}. Now we assume that $t$ is the last time up to which the estimate \ref{bar} holds
and that $l$ represents the scale at which the barrier estimate is breached. That is, we assume
that the inequality \ref{bar} holds at time $t$, but $b_l(t) = C 2^k  2^{{l - k \over 10}},$
and ${db_l \over dt} \geq 0,$ so that it is possible that $b_l$ will be larger at slightly later times.
We now examine the inequality \ref{est2}. We observe that the low-high and high-low terms satisfy the estimate
$$\left( \sum_{s \leq l+5}  b_s(t) \right)
\sum_{\alpha=-5}^5 b_{l+\alpha}(t) \lesssim 2^k b_l(t).$$
Similarly, we see that the high high terms are dominated by a geometric sum:
$$\sum_{s \geq l} 2^{{d(l-s)}} \sum_{\alpha=-5}^5 b_s(t) b_{s+\alpha}(t) \lesssim b_l(t)^2 \lesssim  2^k b_{l}(t).$$
However, by choosing $C$ sufficiently large, we see that $l$ is sufficiently larger than $j_{k,u}$ so that
${2^{{(d+2)l \over 2}} \over l^{\gamma}}$ dominates $2^k$. Thus the negative term in inequality \ref{est2} dominates
the positive term and we have shown ${db_l \over dt} < 0$, a contradiction. \end{proof}

Now in light of the estimates \ref{energy} and \ref{barrier} we see that there is a universal constant $c$ and a scale
$j_{k,d}$ so that
$$\left ( \sum_{j=0}^{j_{k,d}}  + \sum_{j=j_{k,u}+c} \right )  b_k(t) \leq  {2^k \over 10},$$
with 
$$j_{k,u}-j_{k,d}+c \lesssim \log k,$$
as long as $t_k < t < t_{k+1}$. Thus there are at most $\log k$ scales that can contribute to the growth of the 
Besov norm $c(t)$ from $2^k$ to $2^{k+1}$. We pigeonhole to find a scale with a lot of growth. It turns out that
the worst case for us is when the growth occurs near scale $j_{k,u}$. We pick an $\epsilon>0$ sufficiently small
to be determined later so that we are guaranteed that there is some integer $s \geq -c$ so that
 we have, letting $l=j_{k,u}-s$, the quantity  $b_{l} (t)$ increasing by $\gtrsim  2^{k-s\epsilon}$ in the time 
period between $t_k$ and $t_{k+1}$.
Moreover in light of the definition of $c(t)$ and the inequalities \ref{trivial} and \ref{bar}, we can
control the right hand side of the inequality \ref{est2} to obtain that for $t_k < t <t_{k+1}$, we have
$${db_l \over dt}  \lesssim 2^{2k},$$
so that the increase of $b_l(t)$ takes place over a time period of length $\gtrsim 2^{-k-s\epsilon}$.
It turns out that this guarantees us enough dissipation to reach our conclusion.

Recall that from conservation of energy we have that the dissipation is bounded:
  \begin{align} \label{diss} 
&\int_0^{\infty}  \langle Du , Du \rangle  dt  \\
 \sim  &\sum_{j=0}^{\infty} \int_0^{\infty}
{2^{(d+2) j \over 2} \over j^{2 \gamma} }  ||P_j u(t)||_{L^2}^2 dt  \\
\sim &\sum_{j=0}^{\infty} \int_0^{\infty} {2^{-(d+2) j \over 2} \over j^{2 \gamma} } b_j(t)^2  dt \\
\lesssim E^2  \end{align} 

Now we use the fact that $b_l(t)$ is $\gtrsim  {2^{k-s\epsilon}}$ for a time period of  
$\gtrsim {1 \over 2^{k+s\epsilon}}$ to obtain a lower bound on the $l$th term in the third line of \ref{diss}.

We calculate that
$$ {2^{-(d+2) l \over 2} \over l^{2 \gamma} }= {2^{-(d+2) l \over 2} l^{2 \gamma} \over l^{4 \gamma} }
\sim {2^{-k} 2^{(d+2)s \over 2} \over k^{4 \gamma}},$$
where in the final equality we used the definition of $j_{k,u}$ and the fact that $l \sim k$.
Thus estimating the $l$th term of the third line of \ref{diss}, we see that we get 
$\gtrsim {2^{(d+2)s \over 2} \over k^{4 \gamma} 2^{3 \epsilon s}}$ dissipation between $t_k$ and $t_{k+1}$.
Choosing  $\epsilon < {d+2 \over 6}$, we get that the total amount of dissipation between $t_k$ and $t_{k+1}$
is $\gtrsim {1 \over k^{4 \gamma}}$. The sum of these
quantities diverges as long as $\gamma \leq {1 \over 4}$ and so we reach a contradiction. What we have shown is that
the Besov norm $c(t)$ is uniformly bounded.

Now all that remains is the relatively simple and standard task of showing that control on the Besov norm
guarantees that an initial solution remains smooth. For any $\mu>0$, there is a constant $F$ so that
$$b_j(0) \leq F 2^{-\mu j}.$$
Now we observe one more barrier estimate.

\bigskip

\begin{lemma} \label{lastbar} Let $c(t)<M$ for all time $t$. There is a universal constant $K$ so that
for all $j$
\begin{equation} \label{conc} b_j(t) \leq  e^{KMt} F 2^{-\mu j}. \end{equation}
\end{lemma}

\begin{proof}

As before, we assume that the inequality \ref{conc} holds until time $t$ and $b_j(t)= e^{KMt} F 2^{-\mu j}$.
(Note that just the definition of $c(t)$ guarantees us that this quantity is at most $M$.) Then we just examine
the positive terms in the inequality \ref{est2}. We see immediately that
$${d b_j \over dt} \lesssim M b_j(t) + b_j(t)^2  \lesssim M b_j(t).$$
Thus by Gronwall's inequality, the barrier can't be broken.

\end{proof}

In light of the fact that $c(t)$ is uniformly bounded and in light of Lemma \ref{lastbar}, we have proven
Theorem \ref{main}.

\bigskip
\bigskip

\tiny

\textsc{N. KATZ, DEPARTMENT OF MATHEMATICS, INDIANA UNIVERSITY, BLOOMINGTON IN}

{\it nhkatz@indiana.edu}

\bigskip

\textsc{A. TAPAY, DEPARTMENT OF MATHEMATICS, INDIANA UNIVERSITY, BLOOMINGTON IN}

{\it atapay@indiana.edu}


\begin{thebibliography}{5}

\vskip.125in

\bibitem[CS]{CS} A. Cheskidov,R. Shvidkoy {\it A unified approach to regularity problems for the 3D Navier-Stokes and Euler equations: 
the use of Kolmogorov's dissipation range}  arXiv  1102.1944


\bibitem[DKV]{DKV} M. Dabkowski,A. Kiselev, and V. Vicol  {\it Global well-posedness for a slightly supercritical 
surface quasi-geostrophic equation}  arXiv 1106.2137

\bibitem[Pav]{Pav}  N. Pavlovic {\it Use of Littlewood-Paley operators for the equations
of fluid motion}
Ph.D. Thesis, University of Illinois at Chicago, (2002)


\bibitem[Tao]{Tao}  T. Tao {\it Global regularity for a logarithmically supercritical
hyperdissipative
Navier-Stokes equation}  Anal PDE 2 (2009), no. 3  361-366

\bibitem[Tao2]{Tao2} T. Tao {\it Harmonic analysis in the phase plane} UCLA Math 245A
Winter 2001 notes:
http://www.math.ucla.edu/~tao/254a.1.01w/

\bibitem[Tao3]{Tao3} T. Tao {\it Global regularity for a logarithmically supercritical
defocusing nonlinear wave equation for spherically symmetric data} J. Hyperbolic Diff. Equations
4 (2007) 259-266







\end{thebibliography}
\end{document}